\documentclass[reqno]{amsart} 
\usepackage{enumerate,amsmath,amssymb} 
\usepackage{pstricks,pst-node,pst-coil,pst-plot} 
\usepackage{hyperref}

\theoremstyle{plain} 
\newtheorem{step}{Step} 
\newtheorem{thm}{Theorem}[section] 
\newtheorem{theorem}[thm]{Theorem} 
\newtheorem{cor}[thm]{Corollary} 
\newtheorem{corollary}[thm]{Corollary} 
 
\newtheorem{lemma}[thm]{Lemma} 
\newtheorem{prop}[thm]{Proposition} 
\newtheorem{proposition}[thm]{Proposition} 
 
\newtheorem{question}[thm]{Question} 
\theoremstyle{remark}

\theoremstyle{definition}

\def\al{{\alpha}}

\def\De{{\Delta}}
\def\om{{\omega}}

\def\la{{\lambda}}
\let\La\Lambda

\def\si{{\sigma}}

\def\ga{{\gamma}}
\def\epsilon{{\varepsilon}}
\def\ep{{\varepsilon}}

\def\phi{{\varphi}}

\DeclareMathAlphabet{\doba}{U}{msb}{m}{n} 

\gdef\mH{\doba{H}}
\gdef\mN{\doba{N}}
\gdef\mM{\doba{M}}

\gdef\mR{\doba{R}}
\gdef\mS{\doba{S}}

\def\cF{\mathcal{F}}

\def\PSp{{\mathop{\rm PSp}}}
\def\vol{{\mathop{\rm vol}}}
\def\Scal{{s}}

\let\pa\partial 
 
\let\<\langle 
\let\>\rangle 
\def\dvert{{d_{\rm vert}}} 
\def\sh#1{\mathop{\rm sh}\nolimits_{#1}}

\let\ti\tilde
\let\ul\underline

\def\ula{\underline{\lambda}}

\newcommand{\definedas}{\mathrel{\raise.095ex\hbox{\rm :}\mkern-5.2mu=}}

\newcounter{mnotecount}[section]

\def\geqtwo{2+}

\begin{document} 


\title
{Square-integrability of solutions of the Yamabe equation}
 
\author{Bernd Ammann} 
\address{Fakult\"at f\"ur Mathematik \\ 
Universit\"at Regensburg \\
93040 Regensburg \\ 
Germany}
\email{bernd.ammann@mathematik.uni-regensburg.de}

\author{Mattias Dahl} 
\address{Institutionen f\"or Matematik \\
Kungliga Tekniska H\"ogskolan \\
100 44 Stockholm \\
Sweden}
\email{dahl@math.kth.se}

\author{Emmanuel Humbert} 
\address{Laboratoire de Math\'ematiques et Physique Th\'eorique \\ 
Universit\'e de Tours \\
Parc de Grandmont \\
37200 Tours - France \\}
\email{Emmanuel.Humbert@lmpt.univ-tours.fr}

\begin{abstract}
We show that solutions of the Yamabe equation on certain
$n$-dimensional non-compact Riemannian manifolds which are bounded and
$L^p$ for $p=2n/(n-2)$ are also $L^2$. This $L^p$-$L^2$-implication
provides explicit constants in the surgery-monotonicity formula for the
smooth Yamabe invariant in our article
\cite{ammann.dahl.humbert:p08a}. As an application we see that the
smooth Yamabe invariant of any $2$-connected compact $7$-dimensional
manifold is at least $74.5$. Similar conclusions follow in dimension
$8$ and in dimensions $\geq 11$.
\end{abstract}

%

\date{\today}


\maketitle

\tableofcontents

\section{Introduction}

The goal of this article is to show that bounded positive solutions of
the Yamabe equation on certain $n$-dimensional non-compact Riemannian
manifolds $\mM_c^{n,k}$ with finite $L^{p_n}$-norm, $p_n = 2n/(n-2)$, also
have finite $L^2$-norm. For $c\neq 0$ these spaces $\mM_c^{n,k}$ are
products of rescaled hyperbolic space with a standard sphere, while
$\mM_0^{n,k} = \mR^{k+1} \times \mS^{n-k-1}$. The integer $k$ satisfies 
$0 \leq k \leq n-3$. The goal is achieved in some cases, we then say
that the $L^{p_n}$-$L^2$-implication holds. In particular, it holds
for $0\leq k\leq n-4$. If $k=n-3$ and if $n$ is sufficiently large we will
find counterexamples, see Section~\ref{sec.counterex}. 

In the favourable cases the proof of the $L^{p_n}$-$L^2$-implication
is obtained by a combination of tricky integration and suitable
estimates, and there is not much hope to generalize our technique to a
much larger class of non-compact manifolds. The reader will thus
probably ask why we develop such estimates for these special spaces.

The reason is that these spaces appear naturally as limit spaces in 
the surgery construction of our article 
\cite{ammann.dahl.humbert:p08a}, and this limit construction also
provides $L^{p_n}$-solutions of the Yamabe equation.
The main result of \cite{ammann.dahl.humbert:p08a} is a surgery
formula for the smooth Yamabe invariant $\si(M)$, see 
Subsection~\ref{sigma.def} for the definition. The surgery formula
states that if $M$ is a compact manifold of dimension $n$ and if $N$
is obtained from~$M$ through $k$-dimensional surgery, then the smooth
Yamabe invariants~$\si(M)$ and~$\si(N)$ satisfy
\begin{equation}\label{eq.sigma.surg}
\sigma(N) \geq \min(\sigma(M), \Lambda_{n,k}),
\end{equation}
provided that $k\leq n-3$. The definition of the numbers
$\Lambda_{n,k}$ is quite involved, see Subsection~\ref{Lambda.def},
but they are proven to be positive and to depend only on $n$ and~$k$. 
The $L^{p_n}$-$L^2$-implication helps to derive explicit lower bounds
for these numbers, see Subsections~\ref{subsec.role} to \ref{Lambda.def}.

Using estimates for product manifolds from the article
\cite{ammann.dahl.humbert:p11} we obtain explicit positive lower
bounds for $\La_{n,k}$ in the case $2\leq k\leq n-4$, see
Corollary~\ref{maincor} for details. This leads directly to a uniform
positive lower bound for the smooth Yamabe invariant of 
$2$-connected manifolds that are boundaries of compact spin manifolds,
see Corollary~\ref{app.2-conn}. In dimensions 
$n\in\{5,6,7,8\}$ and $n\geq 11$ we obtain as explicit positive lower bound 
$\si(M)$, provided that $M$ is $2$-connected with vanishing index in 
$\al(M)\in KO_n(pt)$, see Corollary~\ref{app.van.ind}.
Using results from \cite{ammann.dahl.humbert:p11c:sigmaexp},
\cite{petean.ruiz:p10}, and \cite{petean.ruiz:p11} an explicit
positive lower bound can also be obtained in the case $n=4$, $k=1$ and
in the case $n=5$, $k \in \{1,2\}$.

\subsection*{Acknowledgements}

B. Ammann was partially supported by the DFG Sachbeihilfe AM 144/2-1.
M.~Dahl was partially supported by the Swedish Research Council. 
E.~Humbert was partially supported by ANR-10-BLAN 0105. 
We want to thank J. Petean for enlightening discussions relating to 
this article.

\section{Preliminaries and Notation}
\label{Section_preliminaries}

\subsection{Spaces of constant curvature and the model spaces
 $\mM_c^{n,k}$ }
\label{model}

Here we fix notation for the spaces of constant curvature and we 
define the spaces $\mM_c^{n,k}$.

We denote the Euclidean metric on~$\mR^n$ by $\xi^n$. The sphere 
$S^n \subset \mR^{n+1}$ equipped with its standard round metric $\rho^n$ 
is denoted by $\mS^n$. We set $\om_\ell \definedas \vol(\mS^\ell)$.
For $c \in \mR$ let $\mH^{k+1}_c$ be the simply connected, connected,
complete Riemannian manifold of constant sectional curvature
$-c^2$. Its metric will be denoted as $\eta_c^{k+1}$. Polar
coordinates around a point $x_0$ gives an identification of 
$\mH^{k+1}_c \setminus \{ x_0 \}$ with $(0,\infty) \times S^k$
under which
\[
\eta_c^{k+1} = dr^2 + \sh{c}(r)^2\rho^k
\]
where 
\begin{equation*}
\sh{c}(r)
\definedas
\begin{cases}
\frac{1}{c} \sinh(cr) & \text{if } c \neq 0 , \\
r & \text{if } c = 0 .
\end{cases}
\end{equation*}

We denote the product metric on~$\mH^{k+1}_c \times \mS^{n-k-1}$ by
\begin{equation*}
G_c \definedas \eta^{k+1}_c + \rho^{n-k-1}
\end{equation*}
and we define the model space $\mM_c^{n,k}$ through
\[
\mM_c^{n,k} \definedas \mH^{k+1}_c \times \mS^{n-k-1}.
\]
The scalar curvature of $\mM_c^{n,k}$ is 
$\Scal^{G_c} = - c^2 k(k+1) + (n-k-1)(n-k-2)$. Note that 
$\mM_c^{n,k}=\mM_{-c}^{n,k}$. As a consequence all following infima
over $c\in[-1,1]$ could be taken over $c\in [0,1]$. From the conformal
point of view the case $c = \pm 1$ is special since $\mM_{\pm 1}^{n,k}
= \mH_{\pm 1}^{k+1} \times \mS^{n-k-1}$ is conformal to 
$\mS^n \setminus \mS^k$, see
\cite[Proposition 3.1]{ammann.dahl.humbert:p08a}.

\subsection{The conformal Yamabe constant}

For integers $n \geq 3$ we set $a_n \definedas \frac{4(n-1)}{n-2}$ and
$p_n \definedas \frac{2n}{n-2}$. For a Riemannian manifold $(M,g)$ 
we denote the scalar curvature by $\Scal^g$, the Laplace operator
$\Delta^g$, and the volume form $dv^g$. The conformal Laplacian is
defined as $L^g \definedas a_n \Delta^g + \Scal^g$. In general the
dependence on the Riemannian metric is indicated by the metric as a
superscript. 

Let $C^\infty_c(M)$ denote the
space of compactly supported smooth functions on~$M$. For a Riemannian
manifold $(M,g)$ of dimension~$n \geq 3$ we define the Yamabe
functional by 
\begin{equation*}
\cF^g (u) 
\definedas 
\frac{
\int_M \left( a_n |du|_g^2 + \Scal^g u^2 \right) \, dv^g}
{\left( \int_M |u|^{p_n} \, dv^g \right)^{\frac{2}{p_n}}} ,
\end{equation*}
where $u \in C^\infty_c(M)$ does not vanish identically. The 
{\it conformal Yamabe constant} $\mu(M,g)$ of $(M,g)$ is defined by 
\begin{equation*}
\mu(M,g) \definedas 
\inf_{u \in C_c^{\infty}(M), u \not\equiv 0} \cF^g(u).
\end{equation*}
Here $M$ is allowed to be compact or non-compact.

If $M$ is compact, then the infimum is attained by a positive smooth
function. It thus satisfies the associated Euler-Lagrange equation,
called the \emph{Yamabe equation},
\begin{equation} \label{eq.conf.1} 
L^g u = \mu u^{p_n - 1}
\end{equation}
for a suitable constant $\mu$.

\subsection{The smooth Yamabe invariant}\label{sigma.def}

Let $M$ be a compact manifold of dimension~$n \geq 3$. The smooth
Yamabe invariant of $M$ defined as 
\begin{equation*}
\sigma(M) \definedas \sup \mu(M,g)
\end{equation*}
where the supremum is taken over the set of all Riemannian metrics
on~$M$. It is known that $\sigma(S^n) = \mu(\mS^n) = n(n-1) \om_n^{2/n}$.

\subsection{The role of the model spaces in the surgery formula}
\label{subsec.role}

To give some background and motivation we will now briefly explain 
the role of the model spaces $\mM_c^{n,k}$ in 
\cite{ammann.dahl.humbert:p08a}, and we want to give a rough 
idea why the invariants $\La_{n,k}$ defined in the following
subsections appear. The discussion in the present subsection is not
needed in the proofs of the following results. We thus try to avoid
technical details, and do not aim for logical completeness.

Assume that $(M,g)$ is a compact Riemannian manifold of dimension
$n\geq 3$ and that~$N$ is obtained by $k$-dimensional surgery from
$M$, where $0 \leq k\leq n-3$. In \cite{ammann.dahl.humbert:p08a} a
sequence of metrics $g_i$ is constructed on~$N$. To prove the surgery
formula one has to show 
\begin{equation}\label{surg.to.show}
\limsup_{i\to\infty} \mu(N,g_i) \geq \min (\mu(M,g),\La_{n,k})
\end{equation}
for a suitable positive constant $\La_{n,k}$.

The solution of the Yamabe problem on $(N,g_i)$ provides positive
smooth functions $u_i\in C^\infty(N)$ satisfying the Yamabe equation
\begin{equation*}
L^{g_i}u = \mu_i u_i^{p_n-1}
\end{equation*}
where $\mu_i=\mu(N,g_i)$ and $\|u_i\|_{L^{p_n}}=1$.

In order to prove \eqref{surg.to.show}, one analyses the ``limits'' of
the functions $u_i$ in various cases. In some cases these functions
$u_i$ ``converge'' to a nontrivial solution of the Yamabe equation 
$L^g u = \mu u^{p_n-1}$ on~$(M,g)$ with $\mu = \limsup \mu_i$, and
it follows that $\limsup_{i\to\infty} \mu(N,g_i) \geq \mu(M,g)$. In
other cases the functions concentrate in some points, and
$\limsup_{i\to\infty} \mu(N,g_i) \geq \mu(\mS^n)$ follows by
``blowing-up'' such points. This means that one suitably rescales
normal coordinates centered in the blow-up point, and then the
solutions of the Yamabe equation converge to a solution on $\mR^n$,
which then yields a solution of the Yamabe equation on a sphere. 
This phenomenon is also often described in the literature by saying 
that ``a sphere bubbles off''. 

However, it can also happen that the functions $u_i$ converge to
a solution of the Yamabe equation on a model space $\mM_c^{n,k}$ with
$|c| \leq 1$, this corresponds to Sub\-ca\-ses~II.1.1 and~II.2 in the
proof of Theorem 6.1 in \cite{ammann.dahl.humbert:p08a}. In this case
one obtains points $x_i \in N$, such that the pointed Riemannian
manifolds $(N,g_i,x_i)$ ``converge'' to $(\mM_c^{n,k},\bar x)$. Here 
``convergence'' means that balls of arbitrary radius $R$ around $x_i$
in $(N,g_i)$ converge for fixed $R$ and $i\to \infty$ in the
$C^\infty$-sense to a ball of radius $R$ around $\bar x$ in
$\mM_c^{n,k}$. The functions $u_i$ will then converge to a positive
solution $\bar u$ of the Yamabe equation~\eqref{eq.conf.1} on
$\mM_c^{n,k}$. The $L^{p_n}$-norm of the limit function does not increase, 
that is $\|\bar u\|_{L^{p_n}}\leq 1$.

This raises the following question. 

\begin{question}
Assume that $\bar u\in C^\infty(\mM_c^{n,k})$ is a positive solution
of the Yamabe equation $L^{G_c} \bar u = \lambda \bar u^{p_n-1}$ with
$0<\|\bar u\|_{L^{p_n}}\leq 1$. Does this imply $\la \geq \mu(\mM_c^{n,k})$?
\end{question}

If $\bar u$ is in $L^2$ (and thus in the Sobolev space $H^{1,2}$), then 
integration by parts 
$\int \bar u \De \bar u \, dv = \int |d \bar u|^2 \, dv$ is allowed 
on $\mM_c^{n,k}$ and it easily follows that that the answer to the 
question is positive. In this case we will say that the
$L^{p_n}$-$L^2$-implication holds.

In turn the $L^{p_n}$-$L^2$-implication implies that 
\eqref{surg.to.show} holds for 
\[
\La_{n,k} \definedas \inf_{c\in[-1, 1]} \mu(\mM_c^{n,k}).
\]

We will find conditions under which the $L^{p_n}$-$L^2$-implication
holds. But we will also obtain examples where it is violated, see
Section~\ref{sec.counterex}. 


The fact that the $L^{p_n}$-$L^2$-implication is violated in some
cases led to a technical difficulty in \cite{ammann.dahl.humbert:p08a} 
which was solved by introducing the constant $\La_{n,k}^{(2)}$ into
the definition of $\La_{n,k}$ in
\cite[Definition~3.2]{ammann.dahl.humbert:p08a}. This is sufficient
for proving the \emph{positivity} of the constant $\La_{n,k}$. 
However, in order to obtain an explicit positive lower bound
for~$\La_{n,k}$ one would like to avoid the constant
$\La_{n,k}^{(2)}$. The possibility to prove the
$L^{p_n}$-$L^2$-implication in some cases was mentioned in 
\cite[Remark~3.4]{ammann.dahl.humbert:p08a}. As a consequence finding
a positive lower bound for $\La_{n,k}$ reduces to finding a positive
lower bound for the constants $\mu(\mM_c^{n,k})$, uniform in 
$c \in [0,1]$. 

In the meantime new results for the explicit lower bounds for
$\mu(\mM_c^{n,k})$ were obtained in \cite{ammann.dahl.humbert:p11} and 
\cite{ammann.dahl.humbert:p11c:sigmaexp}, and a proof of 
\cite[Remark~3.4]{ammann.dahl.humbert:p08a} is needed. The goal of the 
present article is to provide this proof.

\subsection{Modified conformal Yamabe constants}

The technical difficulty described in the previous subsection 
required the introduction of a modified conformal Yamabe constant. 
In fact, two different subcases
require two versions of modified constants, namely the modified
conformal Yamabe constants $\mu^{(1)} (N,h)$ and $\mu^{(2)} (N,h)$,
defined below. Our article aims to give some clarification of the
relation between these invariants for the model spaces $\mM_c^{n,k}$. 

Let $(N,h)$ be a Riemannian manifold of dimension $n$. For $i=1,2$
we let $\Omega^{(i)}(N,h)$ be the set of non-negative $C^2$ functions
$u$ on $N$ which solve the Yamabe equation 
\[
L^h u = \mu u^{p_n - 1}
\]
for some $\mu = \mu(u) \in \mR$. We also require that the functions 
$u \in \Omega^{(i)}(N,h)$ satisfy 
\begin{itemize}
\item[{\rm (a)}] $u \not \equiv 0$,
\item[{\rm (b)}] $\|u\|_{L^{p_n}(N)} \leq 1$,
\item[{\rm (c)}] $u \in L^{\infty}(N)$,
\end{itemize}
together with
\begin{enumerate}
\item[{\rm (d1)}] $u \in L^2(N)$, for $i=1$, 
\end{enumerate}
or
\begin{enumerate}
\item[{\rm (d2)}] $\mu(u) \|u\|^{p_n - 2}_{L^{\infty}(N)} \geq
 \frac{(n-k-2)^2(n-1)}{8(n-2)}$, for $i=2$. 
\end{enumerate}
For $i=1,2$ we set
\begin{equation*} 
\mu^{(i)} (N,h) \definedas \inf_{u \in \Omega^{(i)}(N,h)} \mu(u).
\end{equation*}
In particular $\mu^{(i)}(N,h) = \infty$ if $\Omega^{(i)}(N,h)$ is
empty. If $N$ is compact then the solution of the Yamabe problem
trivially implies $\mu(N,h) = \mu^{(1)}(N,h) = \mu^{(2)}(N,h)$. 

We will use this for $(N,h)=\mM_c^{n,k}$. 
In \cite[Lemma~3.5]{ammann.dahl.humbert:p08a} we already showed 
$\mu^{(1)} (\mM_c^{n,k})\geq \mu (\mM_c^{n,k})$ if $0\leq k\leq n-3$.
In the present article we will show that (b) implies (d1) 
in many cases. As a consequence we will obtain
\begin{equation*}
\mu^{(2)} (\mM_c^{n,k})
\geq \mu^{(1)} (\mM_c^{n,k})
\geq \mu (\mM_c^{n,k})
\end{equation*}
in theses cases. We refer to Theorem~\ref{thlp},
Corollary~\ref{coro1}, and Corollary~\ref{maincor} for details.

\subsection{The numbers $\Lambda_{n,k}$}
\label{Lambda.def}

For integers $n \geq 3$ and $0 \leq k \leq n-3$ set
\begin{equation*}
\La^{(i)}_{n,k} 
\definedas 
\inf_{c \in [-1,1]} 
\mu^{(i)} (\mM_c^{n,k}) 
\end{equation*}
and
\begin{equation*}
\La_{n,k} 
\definedas
\min \left\{ \La^{(1)}_{n,k},\La^{(2)}_{n,k}\right\}.
\end{equation*}
It is not hard to show that $\La_{n,0} = \mu(\mS^n)$,
see \cite[Subsection 3.5]{ammann.dahl.humbert:p08a}. 
The following positivity result for $\La_{n,k}$ is proved 
in~\cite[Theorem 3.3]{ammann.dahl.humbert:p08a}.
\begin{theorem} 
For all $n \geq 3$ and $0 \leq k \leq n-3$, we have 
$\La_{n,k} >0$.
\end{theorem}
Furthermore, the following surgery result is concluded in 
\cite[Corollary~1.4]{ammann.dahl.humbert:p08a}.
\begin{theorem}
Inequality~\eqref{eq.sigma.surg} holds for $0 \leq k \leq n-3$
and the numbers $\La_{n,k} > 0$ defined above.
\end{theorem}

\section{Main Theorem}

\begin{theorem} \label{thlp}
Let $c\in [-1,1]$ and let 
$u \in L^{\infty}(\mM_c^{n,k}) \cap L^{p_n}(\mM_c^{n,k})$
be a smooth positive solution of 
\begin{equation} \label{eqyamabe} 
L^{G_c} u = \mu u^{p_n - 1}. 
\end{equation}
Assume that 
\begin{equation} \label{assumpnk}
2k\,|c| < n(n-k-2),
\end{equation} 
then $u \in L^2(\mM_c^{n,k})$.
\end{theorem} 

Inequality~\eqref{assumpnk} holds when 
\begin{itemize} 
\item $n \leq 5$, $k \in \{0, \cdots , n-3 \}$, and $|c|\leq 1$; or 
\item $n \geq 6$, $k \in \{0, \cdots , n-4 \}$, and $|c|\leq 1$; or 
\item $n = 6$, $k=n-3$, and $|c|<1$.
\end{itemize}
This follows from the fact that $2k \leq n(n-k-2)$ is equivalent to 
$k \leq n - 4 + \frac{8}{n+2}$. 


\begin{cor} \label{coro1}
We have 
\begin{equation*}
\mu^{(2)}(\mM_c^{n,k})
\geq \mu(\mM_c^{n,k})
\end{equation*}
for all $k\leq n-4$. 
The same statement holds for $k=n-3$ and $n\in \{4,5\}$. 
\end{cor} 

\begin{proof}[Proof of Corollary \ref{coro1}] 
Under the conditions of the corollary Assumption~\eqref{assumpnk} 
holds, and hence Theorem~\ref{thlp} implies 
\begin{equation*}
\Omega^{(2)}(\mM_c^{n,k}) 
\subset \Omega^{(1)}(\mM_c^{n,k})
\end{equation*}
and as a consequence we get 
\begin{equation*}
\mu^{(2)}(\mM_c^{n,k})
\geq \mu^{(1)}(\mM_c^{n,k}).\
\end{equation*} 
On the other hand it is proved in 
\cite[Lemma~3.5]{ammann.dahl.humbert:p08a} that 
\begin{equation*}
\mu^{(1)}(\mM_c^{n,k})
\geq
\mu(\mM_c^{n,k})
\end{equation*}
for all $k \in \{ 0,\cdots, k-3\}$. The corollary follows.
\end{proof}

\begin{proof}[Proof of Theorem \ref{thlp}] 
We will now give the proof of Theorem \ref{thlp} in five steps. Let
$u$ be as in the statement of this theorem.

\begin{step} \label{step1} 
The function $u$ tends to $0$ at infinity. 
\end{step} 

We proceed by contradiction and assume that there is an $\ep > 0$ and
a sequence of points $(x_j)_{j \in \mN}$ tending to infinity with $j$ 
such that $u(x_j) \geq \ep$. Denote by $B(x,r)$ the ball of radius $r$
around a point $x$. By taking a subsequence of $(x_j)$ which
tends fast enough to infinity, we can assume that the balls $B(x_j,j)$
are all disjoint. Since $u$ is in $L^{p_n}$ we have
\begin{equation*}
\lim_{j \to \infty} \int_{B(x_j,j)} u^{p_n} \, dv^{G_c} = 0. 
\end{equation*}
Since $\mM_c^{n,k}$ is homogeneous, there are isometries 
$\phi_j : B(x_j,j) \to B(O,j)$ where $O$ is any fixed point in 
$\mM_c^{n,k}$. We now consider the functions 
$v_j \definedas u \circ \phi_j^{-1}$. They are bounded solutions of 
Equation \eqref{eqyamabe} which satisfy 
\begin{equation} \label{lpbounded}
\lim_{j \to \infty} \int_{B(O,j)} v_j^{p_n} \, dv^{G_c} = 0, 
\end{equation}
and
\begin{equation} \label{linfty}
v_j(O) \geq \ep.
\end{equation}

Let $K$ be a compact set containing the point $O$. Since $u$ is
bounded, standard elliptic theory implies that a subsequence of
$(v_j)$ tends to a function $v_K$ on $K$ in~$C^2$. Taking a sequence
$(K_s)$ such that $K_s \subset K_{s+1}$ and 
$\bigcup_s K_s = \mM_c^{n,k}$ we construct successive subsequences
$(v_{j,k_1,\cdots,k_s})$ tending to functions $v_{K_s}$ in $C^2(K_s)$
and such that $v_{K_{s+1}} = v_{K_s}$ on $K_s$. Setting 
$v \definedas v_{K_s}$ on $K_s$, we get a function belonging to
$C^2(\mM_c^{n,k})$. Finally \eqref{lpbounded} and \eqref{linfty} tell
us that 
\begin{equation*}
\int_{ \mM_c^{n,k}} v^{p_n} \, dv^{G_c} = 0 
\end{equation*}
and
\begin{equation*}
v(O) \geq \ep,
\end{equation*}
which gives the desired contradiction. This ends the proof of 
Step \ref{step1}.

We now work in polar coordinates on the hyperbolic space factor of
$\mM_c^{n,k}=\mH_c^{k+1}\times \mS^{n-k-1}$ as introduced in 
Subsection~\ref{model}. We thus study the metric 
$G_c = dr^2 + \sh{c}(r)^2 \rho^{k} + \rho^{n-k-1}$ on the manifold 
$(0,\infty) \times S^k \times S^{n-k-1}$. Using these coordinates,
we denote by $F_r$ the set of constant $r$, that is 
$F_r \definedas S^k \times S^{n-k-1}$ and we denote the restriction of
$g$ to $F_r$ by $g_r = \sh{c}(r)^2 \rho^{k} + \rho^{n-k-1}$. We define
\begin{equation*}
\om(r) \definedas 
\left( \int_{F_r} u^2 \, dv^{g_r} \right)^{\frac{1}{2}}
\end{equation*}
for $r>0$. Next we prove a differential inequality for $\om$.

\begin{step} \label{step2}
For any $\gamma$ with $0 < \gamma < \frac{n-k-2}{2}$ there is an
$r_0(\gamma)$ such that 
\begin{equation*}
\om''(r) \geq \gamma^2 \om(r)
\end{equation*} 
for all $r > r_0(\ga)$.
\end{step} 

The argument for this step is a modification of the proof of Theorem
5.2 in \cite{ammann.dahl.humbert:p08a}.

The Laplacian operators of the total space $\Delta^{G_c}$ and of the
fibers $\Delta^{g_r}$ are related by 
\begin{equation*}
\Delta^{G_c}
=
\Delta^{g_r} - \partial_r^2 + (n-1)H_r \partial_r
\end{equation*}
where $H_r$ denotes the mean curvature of the fiber $F_r$ in
$\mM_c^{n,k}$. It follows that 
\begin{equation*}
\begin{split}
\int_{F_r} u \Delta^{G_c} u \, dv^{g_r} 
&= 
\int_{F_r} 
\left( 
u \Delta^{g_r} u - u \pa_r^2 u + (n-1) H_r u \pa_r u 
\right)
\, dv^{g_r} \\
&=
\int_{F_r} 
\left( 
|\dvert u|^2 - u \pa_r^2 u + (n-1) H_r u\pa_r u 
\right)
\, dv^{g_r}.
\end{split}
\end{equation*}
where $\dvert u$ denotes the differential of $u$ along the fiber, that
is $\dvert u = d(u_{|F_r})$. Using Equation \eqref{eqyamabe} we get
\begin{equation} \label{eqconf1}
a_n \int_{F_r} 
\Big( u \pa_r^2 u - (n-1) H_r u \pa_r u \Big) 
\, dv^{g_r} 
\geq 
\Scal^{G_c} \om(r)^2 - \mu \int_{F_r} u^{p_n} \, dv^{g_r}.
\end{equation}
Computing the derivative of $\om(r)^2/2$ we get 
\begin{equation} \label{eq.udtu}
\begin{split}
\om'(r)\om(r) 
&= 
\frac{1}{2} \frac{d}{dr} \int_{F_r} u^2\, dv^{g_r} \\
&= 
\int_{F_r} u \partial_r u \, dv^{g_r}
- \frac{n-1}{2} H_r \om(r)^2,
\end{split}
\end{equation}
where we used that $H_r$ is constant on $F_r$. Differentiating this
again and using Inequality \eqref{eqconf1} we get 
\begin{equation} \label{w'2}
\begin{split}
\om'(r)^2 + \om''(r) \om(r)
&= 
\int_{F_r} (\partial_r u)^2 \, dv^{g_r} 
+ \int_{F_r} 
\Big(u \partial_r^2 u - (n-1) H_r u \partial_r u \Big)
\, dv^{g_r} \\
&\quad
- \frac{n-1}{2} (\partial_r H_r) \om(r)^2 
- (n-1) H_r \om'(r) \om(r) \\
&\geq
\int_{F_r} (\partial_r u)^2 \, dv^{g_r} 
+ 
\frac{\Scal^{G_c}}{a_n} \om(r)^2 
- 
\frac{\mu}{a_n} \int_{F_r} u^{p_n} \, dv^{g_r} \\
&\quad
- \frac{n-1}{2} (\partial_r H_r) \om(r)^2 
- (n-1) H_r \om'(r) \om(r). \\
 \end{split}
\end{equation}
{}From the Cauchy-Schwarz inequality we get
\begin{equation*}
\begin{split}
\om(r)^2 \int_{F_r} (\partial_r u)^2 \, dv^{g_r}
&\geq
\left( \int_{F_r} u (\partial_r u) \, dv^{g_r} \right)^2 \\
&=
\left( \om'(r)\om(r) + \frac{n-1}{2}H_r \om(r)^2 \right)^2 ,
\end{split}
\end{equation*}
where we uses Equation~\eqref{eq.udtu} in the second line. Thus,
\begin{equation} \label{dtu^2.A}
\int_{F_r} (\partial_r u)^2 \, dv^{g_r}
\geq
\left( \om'(r)+ \frac{n-1}{2} H_r \om(r) \right)^2 .
\end{equation}
Let $\ep > 0$ be a constant to be fixed later. By Step \ref{step1} we
have
\begin{equation} \label{upterm.A}
\frac{\mu}{a_n} \int_{F_r} u^{p_n} \,dv^{g_r} 
\leq
\ep \om(r)^2
\end{equation}
for all $r$ large enough (depending on $\ep$). Inserting 
\eqref{dtu^2.A} and \eqref{upterm.A} into \eqref{w'2} we obtain
\begin{equation*} 
\begin{split}
\om'(r)^2 + \om''(r) \om(r)
&\geq 
\left( \om'(r) + \frac{n-1}{2} H_r \om(r) \right)^2
+ \frac{\Scal^{G_c}}{a_n}
\om(r)^2 - \ep \om(r)^2 \\
&\quad
- \frac{n-1}{2} (\partial_r H_r) \om(r)^2 - (n-1) H_r \om'(r) \om(r),
\end{split}
\end{equation*}
or after some rearranging,
\begin{equation} \label{main3.A}
\om''(r) 
\geq
\underbrace{
\left( \frac{(n-1)^2}{4} H_r^2 + \frac{\Scal^{G_c}}{a_n} 
- \ep - \frac{n-1}{2} (\partial_r H_r) \right) 
}_{ =: \al(r) }
\om(r) .
\end{equation}
A computation tells us that 
\begin{equation*}
H_r 
= 
- \frac{k}{n-1} \partial_r \ln \sh{c}(r)
= 
\begin{cases}
- \frac{k}{n-1} c \coth(cr) & \text{if } c\neq 0, \\
- \frac{k}{n-1} \frac{1}{r} & \text{if } c=0,
\end{cases}
\end{equation*}
so in particular, 
\begin{equation} \label{meancur}
\lim_{r \to \infty} H_r = -\frac{k}{n-1} |c|
\end{equation} 
and 
\begin{equation} \label{meancur'}
\lim_{r \to \infty} \partial_r H_r = 0.
\end{equation}
Using \eqref{meancur} and \eqref{meancur'} together with $\Scal^{G_c}
= -c^2 k(k+1) + (n-k-1)(n-k-2)$ we see that the coefficient $\al(r)$ 
in the right hand side in \eqref{main3.A} tends to 
$\al_c - \ep$ where 
\begin{equation*}
\begin{split}
\al_c 
&\definedas
\frac{(n-1)^2}{4} \frac{k^2}{(n-1)^2} c^2 
+ \frac{n-2}{4(n-1)} \Big( -c^2 k(k+1) + (n-k-1)(n-k-2) \Big) \\
&=
- (n-k-2) \frac{k}{4(n-1)} c^2 + \frac{(n-2)(n-k-1)(n-k-2)}{4(n-1)} \\
&\geq
- (n-k-2) \frac{k}{4(n-1)} + \frac{(n-2)(n-k-1)(n-k-2)}{4(n-1)} \\
&=
\frac{(n-k-2)^2}{4}.
\end{split}
\end{equation*}
Here the inequality comes from the fact the coefficient of
$c^2$ is negative so the smallest value over $c \in [-1,1]$ is
attained for $c=1$. Choosing $\ep$ small enough we have proved
Step \ref{step2}.

By assumption $u$ and thus $\om$ are positive. As a consequence 
we can define
\begin{equation*}
\tau(r) \definedas ( \ln \om(r) )'
\end{equation*}
for $r > 0$.

\begin{step} \label{step3}
One of the following statements is true,
\begin{subequations}
\begin{align}
\liminf_{r\to \infty} \tau(r) &\geq (n-k-2)/2 \label{step3_liminf} , \\
\limsup_{r\to \infty} \tau(r) &\leq -(n-k-2)/2 \label{step3_limsup} .
\end{align}
\end{subequations}
\end{step} 

Assume \eqref{step3_limsup} is not true. Then there is a 
$\ti r \geq r_0(\gamma)$ with $\tau(\ti r) > -(n-k-2)/2$.
We assume $\tau(\ti r)^2 \leq \frac{(n-k-2)^2}4 - 2\ep$ where $\ep>0$. 
Choose $\gamma \definedas \sqrt{\frac{(n-k-2)^2}4 -\ep}$. Using
Step \ref{step2} we calculate 
\begin{equation*}
\tau'(r)
=
\frac{\om''(r)}{\om(r)} - \left(\frac{\om'(r)}{\om(r)}\right)^2
\geq 
\gamma^2 - \tau(r)^2
\end{equation*}
for all $r\geq \ti r$, and thus $\tau'(r)\geq \ep$ as long as 
$\tau(r)^2 \leq \frac{(n-k-2)^2}4 -2 \ep$. An easy argument 
on first order ordinary differential equations yields an $R > \ti r$
such that 
\begin{equation*}
\tau(r) \geq \sqrt{\frac{(n-k-2)^2}4 -2 \ep} 
\end{equation*}
for all $r \geq R$. As $\ep > 0$ can be chosen arbitrarily small we
conclude that \eqref{step3_liminf} must hold, and Step \eqref{step3}
follows.

\begin{step} \label{step4}
If \eqref{step3_liminf} holds then $L^{p_n}$-boundedness of $u$
contradicts the assumption \eqref{assumpnk}.
\end{step} 

Since $\mM^{n,k}_{-c} = \mM^{n,k}_c$ we may assume that $c \geq 0$. At
first we consider the case $c > 0$. In the following argument we
denote by $C$ a positive constant that might change value from line to
line. Since 
\begin{equation*}
\vol^{g_r}(F_r) 
= 
\int_{F_r} \left( \frac{ \sinh(cr)}{c} \right)^{k} 
\, dv^{\rho^k + \rho^{n-k-1}} 
\leq C e^{kcr} 
\end{equation*}
we get 
\begin{equation}\label{om.int.up} 
\begin{split} 
\int_{0}^{\infty} \om(r)^{p_n} e^{-\frac{2kc}{n-2} r} \, dr 
&= 
\int _{0}^{\infty}
\left(\int_{F_r} u^2 \, dv^{g_r} \right)^{\frac{p_n}{2}}
e^{-\frac{2kc}{n-2} r} \, dr \\
&\leq C \int_{0}^{\infty} \int_{F_r} u^{p_n} \, dv^{g_r} dr\\
&\leq C \int_{\mM^{n,k}_c} u^{p_n} \, dv^{G_c}
\end{split}
\end{equation}
using the H\"older inequality. This is bounded since $u$ is
assumed to be in $L^{p_n}(\mM^{n,k}_c)$. If \eqref{step3_liminf} holds
then for any $\gamma \in \left( 0, \frac{n-k-2}{2} \right)$
there is an $r_1 = r_1(\ga)$ so that
\begin{equation*}
\om(r) \geq C e^{\ga r}
\end{equation*}
for all $r \geq r_1$. Thus
\begin{equation} \label{om.int.low}
\int_{0}^{\infty} \om(r)^{p_n} e^{-\frac{2kc}{n-2} r} \, dr 
\geq 
C \int_{r_1}^{\infty} e^{br} \, dr
\end{equation}
where
\begin{equation*}
b \definedas p_n \ga -\frac{2kc}{n-2}= \frac{2}{n-2}(n \ga -kc).
\end{equation*}
If $b \geq 0$ then the right hand side of \eqref{om.int.low} is
infinite, which gives a contradiction to the boundedness
of \eqref{om.int.up}. Thus we have $b < 0$, implying $\ga < kc/n$. 
Taking $\ga \to (n-k-2)/2$ yields $(n-k-2)/2 \leq k|c|/n$, which
finishes the proof of Step \eqref{step4} for $c>0$. 

The case $c=0$ can be solved with similar estimates.

\begin{step}
Conclusion.
\end{step}

It remains to show the $L^2$-boundedness of $u$ if 
\eqref{step3_limsup} of Step~\ref{step3} holds. We choose any 
$\ga \in (0, (n-k-2)/2)$. If \eqref{step3_limsup} holds we have
\begin{equation*}
\om(r) \leq C e^{-\ga r}
\end{equation*}
for all $r \geq r_2(\ga)$, and by possibly enlarging $C$ this holds
for all $r$. From this estimate we have 
\begin{equation*}
\int_{\mM^{n,k}_c} u^2 \, dv^{G_c} 
= 
\int_0^{\infty} \om(r)^2 \, dr < \infty ,
\end{equation*}
which ends the proof of Theorem \ref{thlp}. 
\end{proof}

\section{A counterexample to Theorem \ref{thlp} for $k=n-3$}
\label{sec.counterex} 

As noted after Theorem~\ref{thlp}, Assumption~\eqref{assumpnk}
holds if $n \leq 6$, $k \in \{ 0, \dots, n-3 \}$, $|c|<1$ or 
$n \geq 7$, $k \in \{ 0, \dots, n-4 \}$. It is natural to ask whether
the conclusion of Theorem~\ref{thlp} holds for all $k\in\{0,\ldots,n-3\}$.
The following proposition answers this in the negative.

\begin{prop} \label{cexample}
Let $n \geq 7$. There exists a smooth positive function 
$u \in L^{\infty}(\mH^{n-2}_1 \times \mS^{2}) \cap
L^{p_n}(\mH^{n-2}_1 \times \mS^{2})$ which satisfies 
\begin{equation*}
L^{G_1} u = 0 
\end{equation*}
and is not in $L^2(\mH^{n-2}_1 \times \mS^{2})$. 
\end{prop}

Note that the function $u$ given by Proposition \ref{cexample}
satisfies Equation \eqref{eqyamabe} with $\mu=0$. 

\begin{proof}
Consider $\mS^{n-3}$ as a totally geodesic sphere in $\mS^n$. 
For $y \in \mS^n$ let $\Gamma_y$ be the Green's function of
$L^{\rho^n}$ at $y$. That is $\Gamma_y$ satisfies 
$L^{\rho^n} \Gamma_y = \delta_y$ in the sense of distributions, where
$\delta_y$ is the Dirac distribution at $y$. It is well known that
$\Gamma_y$ exists and satisfies
$\Gamma_y(x) \sim ( 4 (n-1) \om_{n-1} )^{-1} r(x)^{-(n-2)}$ when $x$
tends to $y$. Here $r(x)$ denotes the geodesic distance from $x$ to
$y$. Define $H$ on $\mS^n \setminus \mS^{n-3}$ by
\begin{equation*}
H(x) 
\definedas
\int_{\mS^{n-3}} \Gamma_y (x) \, dv^{\rho^{n-3}}(y).
\end{equation*}
It is straightforward to check that for $x$ tending to $\mS^{n-3}$ 
we have $H(x) \sim c_n' r'(x)^{-1}$ where $c_n'$ depends only on $n$ and
where $r'$ is the geodesic distance to $\mS^{n-3}$. Hence we have 
$H \in L^{p_n}(\mS^n \setminus \mS^{n-3})$ since $n\geq 7$. In 
\cite[Proposition 3.1]{ammann.dahl.humbert:p08a} it was proven that 
$\mS^n \setminus \mS^{n-3}$ and $\mH^{n-2}_1\times \mS^{2}$ are
conformal. Let $f$ be the conformal factor so that 
$G_1 = f^{\frac{4}{n-2}} \rho^n$. As explained in 
\cite{ammann.dahl.humbert:p08a}, $f(r') \sim (r')^{-\frac{2}{n-2}}$ when
$r'$ tends to $0$. We set $u \definedas f^{-1} H$. By conformal
covariance of the conformal Laplacian we have
\begin{equation*}
L^{G_1} u = 0.
\end{equation*}
Moreover, 
\begin{equation*}
\int_{\mH^{n-2}_1 \times \mS^{2}} u^{p_n} \, dv^{G_1} 
= \int_{\mS^n \setminus \mS^{n-3}} H^{p_n} \, dv^{\rho^n}
< \infty.
\end{equation*}
In addition, using the asymptotics of $f$ given above, it is easy to
check that $u$ is not in $L^2$ and hence provides the desired
counterexample.
\end{proof}


\section{Consequences for the surgery formula}
\label{sec.surg.app}

The goal of this paper is to find explicit lower bounds for
$\La_{n,k}$. We find the following.

\begin{corollary} \label{maincor}
Assume that $k \in \{ 2, \cdots, n-4\}$, then 
\begin{equation*}
\Lambda_{n,k} \geq \ul{\Lambda}_{n,k} 
\end{equation*}
where
\[
\ul{\Lambda}_{n,k} 
\definedas
\frac{n a_n}
{((k+1) a_{k+1})^{\frac{k+1}{n}} ((n-k-1) a_{n-k-1})^{\frac{n-k-1}{n}}}
\mu(\mS^{k+1})^{\frac{k+1}{n}} 
\mu(\mS^{n-k-1})^{\frac{n-k-1}{n}}.
\]
\end{corollary}

Note that we have $\Lambda_{n,0} = \mu(\mS^n)$ from
\cite[Section 3.5]{ammann.dahl.humbert:p08a}, and hence the only cases
not covered by this corollary are $k = 1$ and $k = n-3$. Further, 
\[
\ul{\La}_{n,2} = \ul{\La}_{n,n-4} 
= 
n a_n \left(\frac{\pi^2}{12}\right)^{\frac{3}{n}} 
\left(\frac{\mu(\mS^{n-3})}{(n-3) a_{n-3}}\right)^{\frac{n-3}{n}}
=
n a_n\left(\frac{\pi^2}{12}\right)^{\frac{3}{n}} \nu_{n-3}^{1/n},
\]
where we defined
\[
\nu_\ell \definedas
\left( \frac{\mu(\mS^\ell)}{\ell a_\ell} \right)^\ell
= \om_\ell^2\left(\frac{\ell-2}{4}\right)^\ell
\]
and it holds that
\[
\om_\ell 
= \vol(\mS^\ell)
= \frac{2 \pi^{(\ell+1)/2}}{\Gamma\left(\frac{\ell+1}{2}\right)}.
\] 
We define 
\[
\ul{\La}_{n,\geqtwo} \definedas \min\{\ul{\La}_{n,2},\ldots,\ul{\La}_{n,n-4}\}.
\]
Some values for $\ul{\La}_{n,\geqtwo}$ are listed in Figure~\ref{fig.la.table}.
Numerically we calculated  $\ul{\La}_{n,\geqtwo} = \ul{\La}_{n,2}$ for 
$n\leq 3000$, and it seems reasonable to conjecture this for all $n$,
but we do not have a proof.

\begin{proof}[Proof of Corollary \ref{maincor}]
The conformal Yamabe invariant $\mu( \mM_c^{n,k})$ as
defined in Section \ref{Section_preliminaries} for non-compact
manifolds is, by virtue of \cite[Theorem 2.3]{ammann.dahl.humbert:p11}, 
bounded from below by $\ul{\Lambda}_{n,k}$. To see this we just have to 
notice that $\mu(\mH_c^{k+1}) = \mu(\mS^{k+1})$, which holds since 
$\mH_c^{k+1}$ is conformal to a subset of $\mS^{k+1}$. Thus 
Corollary~\ref{maincor} is a direct consequence of 
Corollary~\ref{coro1}. 
\end{proof}

\section{Topological applications}

In this section we derive some topological consequences of our main 
theorem. Recall that by definition a manifold $M$ is $k$-connected, 
$k\geq 1$, if it is connected and if 
$\pi_1(M)=\pi_2(M)=\dots=\pi_k(M)=0$.

\begin{proposition}
Let $M_0$ and $M_1$ be non-empty compact $2$-connected manifolds of 
dimension~$n\geq 7$, and assume that $M_0$ is spin bordant to $M_1$.
Then $M_1$ can be obtained from $M_0$ by a sequence of surgeries of 
dimensions $\ell$, $3\leq \ell\leq n-4$. 
\end{proposition}

Note that $2$-connected manifolds are orientable and spin, and they 
carry a unique spin structure. 

The proposition is well-known, but for the sake of being self-contained 
we include a proof following the lines of~\cite[Lemma 4.2]{labbi:97}). 
As a first step we prove a lemma.

\begin{lemma}
Let $M_0$ and $M_1$ compact spin manifolds of dimension~$n\geq 7$
and assume that $M_0$ is spin bordant to $M_1$. Then there is a 
$3$-connected spin bordism~$W$ from $M_0$ to $M_1$.
\end{lemma}

\begin{proof}[Proof of the Lemma]
We start with a given spin bordism $W_0$ from $M_0$ to $M_1$. From 
this bordism we construct a bordism $W$ which is $3$-connected.

By performing $0$-dimensional surgeries at the bordism, one can 
modify the original bordism $W_0$ to be connected. This can be done 
such that the bordism $W_1$ thus obtained is again orientable, and 
$W_1$ then carries a spin structure.

We now perform $1$-dimensional surgeries to reduce the fundamental 
group to the trivial group. Assume that $[\ga] \in \pi_1(W_1)$. 
We can assume that $\gamma: S^1 \to W_1$ is an embedding. Its normal 
bundle is trivial as $W_1$ is orientable. Performing a $1$-dimensional 
surgery along $\gamma$ using a trivialization $\nu$ of this normal 
bundle yields a new bordism $W_1^{\ga,\nu}$ which depends both on 
$\ga$ and $\nu$. This bordism is orientable. For any $\ga$ one can 
choose a trivialization $\nu$ such that the bordism $W_1^{\ga,\nu}$ 
carries a spin structure that coincides with the spin structure of 
$W_1$ outside a tubular neighborhood of the image of $\gamma$. 
The van Kampen Theorem gives a surjective homomorphism 
$\pi_1(W_1) \to \pi_1(W_1^{\ga,\nu})$ such that $[\gamma]$ is in the 
kernel. The fundamental group $\pi_1(W)$ is finitely generated, 
let $\ga_i$ be disjoint embedded circles such that 
$[\gamma_1], \ldots, [\gamma_\ell]$ is a set of generators $\pi_1(W)$. 
Performing $1$-dimensional surgeries along the $\ga_i$ with suitable 
trivializations of their normal bundles then yields a simply-connected 
spin bordism $W_2$ from $M_0$ to $M_1$.

Next we perform $2$-dimensional surgeries to remove $\pi_2(W_2)$. 
Assume that $[\sigma] \in \pi_2(W_2)$ is given and assume that 
$\sigma: S^2 \to W_2$ is an embedding. Since $W_2$ is spin the normal 
bundle of the image of $\sigma$ is trivial. Performing a $2$-dimensional 
surgery along $\sigma$ yields a new spin bordism $W_2^\si$ which depends 
on the choice of $\si$. However, it is independent of the choice of 
trivialization as different trivializations are homotopic. 
After a finite number of $2$-dimensional surgeries we obtain a 
$2$-connected spin bordism $W_3$ from $M_0$ to $M_1$. 

In a similar way one can also remove $\pi_3(W_3)$. 
The Whitney embedding theorem implies that any class in $\pi_3(W_3)$ 
can be represented by an embedding $\tau:S^3\to W_3$, as $n\geq 6$.
The normal bundle of the image of $\tau$
is trivial, as $\pi_2(O(n-3))=0$. A surgery along $\tau$ with any 
trivialization $\nu$ will then kill $[\tau]$, and 
since $n\geq 7$ the spin bordism~$W_3^{\tau,\nu}$ thus obtained is again 
$2$-connected and 
will have $\pi_3(W_3^{\tau,\nu})\cong \pi_3(W_3)/[\tau]$.
After finitely many surgery steps we obtain a $3$-connected bordism $W$ 
as claimed in the lemma.
\end{proof}

\begin{proof}[Proof of the Proposition]
Assume that $W$ is $3$-connected spin bordism from $M_0$ to $M_1$.
Then $H_i(W,M_j)=0$ for $i=0,1,2,3$, in particular $b_i(W,M_0)=0$ for
these numbers $i$.
We can apply \cite[VIII Theorem~4.1]{kosinski:93} for $k=4$ and $m=n+1$.
One obtains that there is a presentation of the bordism such that for any $i<4$
and any $i>n-3$ the number of $i$-handles is given by $b_i(W,M_0)$. 
Any $i$-handle corresponds to a surgery of dimension $i-1$. 
It remains to show that $b_i(W,M_0)=0$ for $i\in\{0,1,2,3,n+1,n,n-1,n-2\}$.
For $i\in \{0,1,2,3\}$ this was discussed above. By Poincar\'e duality 
$H^{n+1-i}(W,M_0)$ is dual to $H_i(W,M_1)$ which vanishes for $i=0,1,2,3$.
One the other hand the universal coefficient theorem tells us that 
the free parts of $H^i(W,M_0)$ and $H_i(W,M_0)$ are isomorphic. 
Thus $b_i(W,M_0)$ which is by definition the rank of (the free part of) 
$H_i(W,M_0)$ vanishes for $i\in\{n+1,n,n-1,n-2\}$.
\end{proof}

\begin{corollary} \label{app.2-conn}
Let $M$ be a $2$-connected compact manifold of dimension $n\geq 7$
which is a spin boundary. Then 
\[
\si(M) \geq \ul{\La}_{n,\geqtwo}
\]
where $\ul{\La}_{n,\geqtwo}$ is defined in Section \ref{sec.surg.app}.
\end{corollary}

\begin{proof}
Assume that $M$ is the boundary of a compact spin manifold $W$ of dimension
$n+1$. By removing a ball we obtain a spin-bordism from $S^n$ to $M$. 
The preceding proposition tells us that $M$ can be obtained by surgeries
of dimensions $\ell\in\{3,\ldots,n-4\}$ from $S^n$.
By applying the surgery formula \eqref{eq.sigma.surg} and 
Corollary \ref{maincor} we get the stated lower bound for $\si(M)$. 
\end{proof}

\begin{theorem}[{Stolz \cite[Theorem~B]{stolz:92}}]
Let $M$ be a compact spin manifold of dimension $n\geq 5$. Assume 
that the index $\al(M) \in KO_n(pt)$ vanishes. Then $M$ is spin-bordant 
to the total space of an $\mH P^2$-bundle over a base $Q$ for which the 
structure group is $\PSp(3)$.
\end{theorem}

The base $Q$ has to be understood as a spin manifold, so that it admits
a spin structure and the choice of spin structure matters. The theorem 
includes the fact the every spin manifold of dimension $5$, $6$, or $7$ 
is a spin boundary, in these cases $Q = \emptyset$. 

\begin{proposition}[Extended Stolz theorem]
In the case $n \geq 9$ one can assume that $Q$ is connected, and in 
the case $n \geq 11$ one can assume that it is simply connected.
\end{proposition}

Note that $M \definedas \mH P^2 \amalg \mH P^2$ is an $8$-dimensional 
example where $Q$ cannot be chosen to be connected. 
This follows from the fact that $\mH P^2$ has non-vanishing 
signature and thus $[\mH P^2]$ is an element of infinite order in 
$\Omega_8^{\rm spin}$.
If $S^1$ carries the spin structure that does not bound a disc, 
then $\mH P^2\times S^1$ and $\mH P^2 \times S^1 \times S^1$ are 
examples of dimension $9$ and 
$10$ where $Q$ cannot be chosen to be simply connected.
This is a consequence of the fact 
that  $[\mH P^2\times S^1]\in \Omega_9^{\rm spin}$ and 
$[\mH P^2 \times S^1 \times S^1]\in \Omega_{10}^{\rm spin}$ are non-zero 
elements (of order $2$), see
\cite[Cor.~1.9]{anderson.brown.peterson:66} or 
\cite[Cor.~2.6]{anderson.brown.peterson:67}.

\begin{proof}
Assume that $M$ is spin bordant to a spin manifold $N_0$ which is the 
total space of a fiber bundle with fiber $\mH P^2$ and structure group 
$\PSp(3)$ over a base $Q_0$ of dimension $n-8 \geq 1$. By performing 
$0$-dimensional surgery on $Q_0$ we obtain a connected space $Q_1$. 
The spin bordism from $Q_0$ to $Q_1$ which is associated to the 
$0$-dimensional surgeries yields a spin bordism from $N_0$ to a total 
space of a fiber bundle with fiber $\mH P^2$ and structure group 
$\PSp(3)$ over $Q$. This $Q$ is connected, but not necessarily 
simply-connected.

Now assume $n \geq 11$. Any path $\gamma:S^1 \to Q_1$ is homotopic 
to an embedding and has a trivial normal bundle as $Q_1$ is 
orientable. A tubular neighborhood of the image of $\gamma$ in $Q_1$ 
is diffeomorphic to $S^1 \times B^{n-9}$. Any trivialization of this 
normal bundle yields the germ of such a diffeomorphism up to isotopy. 
Because of our condition $n\geq 11$ we can choose the trivialization of
the normal bundle such that the induced spin structure on 
$S^1 \times B^{n-9}$ is the bounding spin structure of $B^2 \times B^{n-9}$. 
Doing a surgery along $\gamma$ with respect to such a trivialization 
we obtain a spin manifold $Q_2$, and the associated bordism from $Q_1$ 
to $Q_2$ is a spin bordism. As $\PSp(3)$ is connected the 
$\mH P^2$-bundle with structure group $\PSp(3)$ extends to a 
bundle of the same type over this bordism.

We now perform a sequence of such $1$-dimensional surgeries, where 
$\gamma$ runs through a generating set of $\pi_1(Q_1)$. The space thus
obtained is then simply-connected.
\end{proof}

Combining the previous results we obtain the following.

\begin{corollary}
Let $M$ be a $2$-connected compact manifold of dimension $n=7$. Then
\[
\si(M)\geq \ul{\La}_{7,\geqtwo}> 74.5.
\]
\end{corollary}
To derive a similar result for $n=8$, we remark that the conformal 
Yamabe constant of $\mH P^2$, equipped with the standard metric, is 
$128\pi/120^{1/4} = 121.4967... > \ul{\La}_{8,\geqtwo} = 92.24278...$.
 
\begin{corollary}\label{app.van.ind}
Let $M$ be a $2$-connected compact manifold of dimension $n=8$.
Then $\si(M)  =  0 $ if $\al(M)\neq 0$, and 
 \[\si(M) \geq \ul{\La}_{8,\geqtwo} > 92.2\] 
if $\al(M)= 0$.
%
\end{corollary}

\begin{proposition}
Let $M_0$ be the total space of a bundle with fiber $\mH P^2$ and
structure group $\PSp (3)$ over a base $B$ of dimension $n-8$. 
Then, if $n\geq 11$ 
\[
\si(M_0) \geq 
\ul{\la}_n 
\definedas 
n a_n \left(\frac{3^6 2^{18}}{7^8 5^2}\pi^8\right)^{1/n} \nu_{n-8}^{1/n}
\]
\end{proposition}

\begin{proof}
M. Streil has shown in \cite{streil:p12} that $\si(M_0)\geq \mu(\mH
P^2\times \mR^{n-8})$, where $\mH P^2\times \mR^{n-8}$ carries 
the product metric of the standard metrics on $\mH P^2$ and
$\mR^{n-8}$. On the other hand it follows from 
\cite[Theorem~2.3]{ammann.dahl.humbert:p11}
that 
\[
\mu(\mH P^2\times \mR^{n-8})
\geq 
\frac{n a_n}{(8 a_8)^{8/n}((n-8)a_{n-8})^{(n-8)/n}} 
\mu(\mH P^2)^{8/n} \mu(\mS^{n-8})^{(n-8)/n}.
\]
On the other hand 
\[
\left(\frac{\mu(\mH P^2)}{8a_8}\right)^8
= \frac{3^6 2^{18}}{7^8 5^2}\pi^8
= 1.32599...\pi^8
= 12581.78...
\]
This clearly implies the proposition.
\end{proof}

As an example we study $n=11$. Then $\nu_3 = \pi^6/8$ and thus 
$\ul{\la}_{11} = 178.23277$. Some further values for 
$\ul{\la}_{n}$ are listed in Figure~\ref{fig.la.table}.

\begin{proposition}
Let $M$ be a $2$-connected compact manifold of dimension $n\geq 11$.
Then $\si(M)  =  0 $ if $\al(M)\neq 0$.
If $\al(M)= 0$, then
\[
 \si(M) \geq
   \begin{cases}
    \ula_{11}> 135.90 & \text{if } n=11, \\
    \ula_{12}> 158.72 & \text{if } n=12, \\
    \ul{\La}_{n,\geqtwo}  & \text{if } n\geq 13.
  \end{cases} 
\]
%
\end{proposition}

\begin{proof}
If $\al(M)= 0$, then we have seen
\[ \si(M)  \geq \min(\ul{\La}_{n,\geqtwo}, \ul{\la}_n) .\]
It remains to compare $\ul{\La}_{n,\geqtwo}$ and $\ul{\la}_n$.
Numerically we calculated  $ \ul{\la}_{11}\leq \ul{\La}_{11,\geqtwo}$,  
$\ul{\la}_{12}\leq \ul{\La}_{12,\geqtwo}$, and 
$\ul{\la}_n \geq \ul{\La}_{n,\geqtwo}$ for $13\leq n\leq 5000$.

For $n\geq 1100$ we found
\[\ula_n^n\geq 1.43 \,\ul{\La}_{n,2}^n\geq 1.43 \,\ul{\La}_{n,\geqtwo}^n\]
by studying the $\Gamma$-function and by using
$\Gamma(n)/\Gamma(n-1/2)\geq \sqrt{n-1}$. This yields the required result.
\end{proof}

\begin{figure} 
\begin{tabular}{c|ccc}
$n$ & $Y(\mS^n)$ & $\ul{\La}_{n,\geqtwo}$ & $\ul{\la}_n$\\
\hline
7 & 113.5272754 & 74.50435 \\
8 & 130.7157953 & 92.24278367 \\
9 & 147.8778709 & 109.4260421 \\
10 & 165.0220642 & 126.4134026 \\
11 & 182.1536061 & 143.3280094 & 135.9033973\\
12 & 199.2758713 & 160.2189094 & 158.7256737\\
13 & 216.3911332 & 177.1071517 & 178.0562033\\
14 & 233.5009793 & 194.0019409 & 196.2714765\\
15 & 250.6065514 & 210.9071013 & 213.9967504\\
16 & 267.7086915 & 227.8239126 & 231.4689436\\
17 & 284.8080344 & 244.7524346 & 248.7967717\\
18 & 301.9050675 & 261.6921542 & 266.0365304\\
\end{tabular}
\caption{Some values for $\ul{\La}_{n,\geqtwo}$ and $\ul{\la}_n$.}
\label{fig.la.table}
\end{figure}


\providecommand{\bysame}{\leavevmode\hbox to3em{\hrulefill}\thinspace}
\providecommand{\MR}{\relax\ifhmode\unskip\space\fi MR }
\providecommand{\MRhref}[2]{%
  \href{http://www.ams.org/mathscinet-getitem?mr=#1}{#2}
}
\providecommand{\href}[2]{#2}

\end{document}